\documentclass[11pt]{amsart}
 \usepackage{amsmath,amssymb}
 \usepackage{xcolor}

\newtheorem{theorem}{Theorem}[section]
\newtheorem{thm}[theorem]{Theorem}
\newtheorem{lemma}[theorem]{Lemma}
\newtheorem{lem}[theorem]{Lemma}
\newtheorem{prop}[theorem]{Proposition}
\newtheorem{coro}[theorem]{Corollary}
\newtheorem{corollary}[theorem]{Corollary}

\theoremstyle{definition}

\newtheorem{defi}[theorem]{Definition}

\newtheorem{exa}[theorem]{Example}

\newtheorem{remark}[theorem]{Remark}

 \theoremstyle{plain}
\newtheorem*{namedthm}{\namedthmname}
\newcounter{namedthm}

\makeatletter
\newenvironment{named}[1]
  {\def\namedthmname{#1}%
   \refstepcounter{namedthm}%
   \namedthm\def\@currentlabel{#1}}
  {\endnamedthm}

 \newcommand{\B}{\mathbb B}
 
 \newcommand{\R}{\mathbb R}
 
 \newcommand{\C}{\mathbb C}
  \newcommand{\PP}{\mathbb P}
 \newcommand{\N}{\mathbb N}

 \newcommand{\e}{\varepsilon}

 \newcommand{\capa}{{\rm Cap}}

 \newcommand{\f}{\varphi}
 
 \newcommand{\p}{\psi}

 \newcommand \al {\alpha}

 \newcommand \psh {{\rm PSH}}
 \newcommand \PSH {{\rm PSH}}
  \newcommand \SH {{\rm SH}}

 \newcommand \MA {{\rm MA}}

 \newcommand \Ent{{\rm Ent}}

 \numberwithin{equation}{section}

 \usepackage{hyperref}
\hypersetup{
    unicode=false,        
    pdftoolbar=true,      
    pdfmenubar=true,       
    pdffitwindow=false,     
    pdfstartview={FitH},    
    pdftitle={Finite entropy vs finite energy},    
    pdfauthor={Di Nezza, Guedj, Lu},     
    colorlinks=true,       
   linkcolor=blue,          
    citecolor=blue,        
    filecolor=black,      
    urlcolor=blue}



\begin{document}

\title[Finite entropy vs finite energy]{Finite entropy vs finite energy}
\author{Eleonora Di Nezza, Vincent Guedj, Chinh H. Lu}

\thanks{The authors are partially supported by the ANR project CRACK}

\address{Institut de Math\'ematiques de Jussieu \\ Sorbonne Universit\'e \\ 4 place Jussieu \\  75005 Paris, France}

\email{\href{mailto:eleonora.dinezza@imj-prg.fr}{eleonora.dinezza@imj-prg.fr}}
\urladdr{\href{https://sites.google.com/site/edinezza/home}{https://sites.google.com/site/edinezza/home}}

\address{Institut de Math\'ematiques de Toulouse   \\ Universit\'e de Toulouse \\
118 route de Narbonne \\
31400 Toulouse, France\\}

\email{\href{mailto:vincent.guedj@math.univ-toulouse.fr}{vincent.guedj@math.univ-toulouse.fr}}
\urladdr{\href{https://www.math.univ-toulouse.fr/~guedj}{https://www.math.univ-toulouse.fr/~guedj/}}

\address{Universit\'e Paris-Saclay, CNRS, Laboratoire de Math\'ematiques d'Orsay, 91405, Orsay, France.}

\email{\href{mailto:hoang-chinh.lu@universite-paris-saclay.fr}{hoang-chinh.lu@universite-paris-saclay.fr}}
\urladdr{\href{https://www.imo.universite-paris-saclay.fr/~lu/}{https://www.imo.universite-paris-saclay.fr/~lu/}}
\date{\today}

 \begin{abstract}
 Probability measures with either finite Monge-Amp\`ere energy or finite  entropy  have played a central role in recent developments
in K\"ahler geometry.
In this note we make a systematic study of quasi-plurisubharmonic potentials 
whose Monge-Amp\`ere measures have finite entropy. 
We show that these potentials  belong to the finite energy class ${\mathcal E}^{\frac{n}{n-1}}$,
where $n$ denotes the complex dimension, and
 provide examples showing that this critical exponent is sharp. 
Our proof relies on refined
Moser-Trudinger  inequalities
for quasi-plurisubharmonic functions. 
\end{abstract}

 \maketitle


\section*{Introduction}

Probability measures with either finite energy \cite{BEGZ10,BBGZ13} or finite entropy \cite{BBEGZ19} have played an important role in recent developments
in K\"ahler geometry (see \cite{BBJ15, BDL16, BDL17, ChCh1, ChCh2, DonICM18, BBEGZ19} and the references therein).

Indeed the search for K\"ahler-Einstein metrics on Fano manifolds  boils down to maximizing the Ding functional whose leading term is a Monge-Amp\`ere energy, while a constant scalar curvature K\"ahler metric minimizes the Mabuchi functional, whose leading term is an entropy.
The purpose of this note is to systematically compare these two notions.

\smallskip

Let $(X,\omega)$ be a compact K\"ahler manifold of complex dimension $n \geq 1$, normalized so that
$$
{\rm Vol}_{\omega}(X):=\int_X \omega^n =1.
$$

We consider $\mu=f \omega^n$, $0 \leq f$,  a probability measure with finite entropy
$$
0 \leq {\rm Ent}_{\omega^n} (\mu):=\int_X f \log f\; \omega^n <+\infty.
$$
 Since $\mu$ is absolutely continuous with respect to the volume form $\omega^n$, it is in particular ``non-pluripolar"
 hence it follows from \cite{GZ07,DiwJFA09} that there exists a unique 
 full mass potential $\f \in {\mathcal E}(X,\omega)$ such that
 $\sup_X \f=0$ and
 $$
 (\omega+dd^c \f)^n=\mu.
 $$
 
 Here $d=\partial+\overline{\partial}$ and $d^c=\frac{i}{2\pi } \left( \partial-\overline{\partial} \right)$ are real operators so that
 $dd^c = \frac{i}{\pi} \partial \overline{\partial}$,  and
 ${\mathcal E}(X,\omega)$ denotes the set of $\omega$-plurisubharmonic functions $\f$
 whose non pluripolar Monge-Amp\`ere measure $(\omega+dd^c \f)^n$ is a probability measure.
 We refer the reader to Section \ref{sec:prelim} for a precise definition.
 We 
 consider, for $p>0$,
 $$
 {\mathcal E}^p(X,\omega):=\{ \f \in {\mathcal E}(X,\omega)\; | \;  E_p(\varphi)<+\infty \},
 $$
 where $E_p(\varphi):= \int_X |\varphi|^p (\omega+dd^c \varphi)^n$.
 
  It has been observed in \cite[Theorem 2.17]{BBEGZ19} that 
 $$
\Ent(X,\omega) \subset  {\mathcal E}^1(X,\omega),
 $$
 and the injection $\Ent(X,\omega) \hookrightarrow  {\mathcal E}^1(X,\omega)$ is compact, 
where $\Ent(X,\omega)$ is the set of $\omega$-psh functions whose Monge-Amp\`ere measure has finite entropy.
 However all computable examples suggest that $\f$ actually belongs to a higher energy class ${\mathcal E}^p(X,\omega)$ for some $p>1$ depending on the dimension. 
 We confirm this experimental observation by showing the following:

\begin{named}{Theorem A} 
Let $\mu=(\omega+dd^c \f)^n=f \omega^n$ be a probability measure
with finite entropy
 $
  {\rm Ent}_{\omega^n} (\mu)=\int_X f \log f \omega^n<+\infty.
  $
    Then
     $$ 
     \f \in  {\mathcal E}^{\frac{n}{n-1}}(X,\omega).
  $$
 Moreover the inclusion ${\rm Ent}(X,\omega) \hookrightarrow {\mathcal E}^{p}(X,\omega)$ is compact for any $p<\frac{n}{n-1}$. 

 This exponent is sharp when $n \geq 2$.  If $n=1$ then $\f$ is   continuous, hence it belongs to
${\mathcal E}^p(X,\omega)$ for all $p>0$.
\end{named}

The case of Riemann surfaces  deserves a special treatment:
 finite entropy potentials turn out to be bounded (and even continuous), but this is no longer the case in higher dimension. The proof of Theorem A   relies on a Moser-Trudinger inequality which provides a strong integrability property of finite energy potentials. This is the content of our second main result:

 \begin{named}{Theorem B}
 Fix $p>0$. There exist positive constants $c,C>0$  depending on 
 $ X,\omega, n,p$ such that, for all $\varphi\in \mathcal{E}^p(X,\omega)$ with $\sup_X \varphi=-1$,   
 	\[
 	\int_X \exp \left ( c |E_p(\varphi)|^{-1/n} |\varphi|^{1+ \frac{p}{n}} \right ) \omega^n \leq C. 
 	\]
\end{named} 
%
\medskip 

Theorem B is an interesting variant of Trudinger's inequality on compact K\"ahler manifolds.
 The case $p=1$ settles a conjecture of Aubin (called Hypoth\`ese fondamentale \cite{Aub84}) which is motivated by the search for K\"ahler-Einstein metrics on Fano manifolds.  The conjecture was previously proved by Berman-Berndtsson \cite{BB11} under the assumption that the cohomology class of $\omega$ is the first Chern class of an ample holomorphic line bundle. 

 We also establish  local versions of these results, valid in any bounded hyperconvex domain of $\C^n$.

 \medskip 
 
 \noindent {\bf Organization of the paper.} 
  We recall the definition of finite energy classes in Section \ref{sec:prelim} where we also give explicit examples 
  of finite entropy potentials. We then establish a Moser-Trudinger inequality in Section \ref{sect: MT inequality}, proving Theorem B.  We show in Proposition \ref{prop: Lmu continuous} that finite entropy measures act continuously on $\PSH(X,\omega)$ endowed with the $L^1$-topology. In  the special case of compact Riemann surfaces  the latter is equivalent to the potential being continuous (see Section \ref{sect: dimension 1}).  Theorem A  will be proved in Section \ref{sect: higher dimension}. 
   We finally treat the case of bounded hyperconvex domains by 
  analyzing the size of the capacity of sublevel sets
   in Section \ref{sec:local}, where we also briefly discuss an extension of 
   the celebrated Moser-Trudinger-Adams inequalities.

\medskip

 \noindent {\bf Acknowledgement.}    We thank S\'ebastien Boucksom and Mattias Jonsson for raising this interesting question. 
 We also thank Bo Berndtsson 
 for useful discussions on this subject. We thank Ahmed Zeriahi for carefully reading the first version of the note and giving numerous fruitful comments.

\section{Preliminaries}\label{sec:prelim}

In the whole paper $(X,\omega)$ is a compact K\"ahler manifold of complex dimension $n\in \N^*$.
We assume $\omega$ is normalized so that $\int_X \omega^n=1$.

 \subsection{Envelopes of quasi-psh functions}
Recall that a function is quasi-pluri\-subharmonic (\emph{qpsh} for short) if it is locally given as the sum of  a smooth and a psh function.   Quasi-psh functions
$\f:X \rightarrow \R \cup \{-\infty\}$ satisfying
$$
\omega+dd^c \f \geq 0
$$
in the weak sense of currents are called $\omega$-psh functions. 

In particular quasi-psh  functions are upper semi-continuous and Lebesgue-integrable.
They are actually in $L^p$ for all $p \geq 1$, and the induced topologies are all equivalent.
%

\begin{defi}
We let $\psh(X,\omega)$ denote the set of all $\omega$-psh functions which are not identically $-\infty$.  
\end{defi}

The set $\psh(X,\omega)$ is a closed subset of $L^1(X)$
for the $L^1$-topology. It was proved by Demailly \cite{Dem92} (see also \cite{BK07}) that 
any $\omega$-psh function
can be approximated by a decreasing sequence of smooth $\omega$-psh functions.

\begin{defi}
A set $P \subset X$ is called pluripolar if $P \subset \{\f=-\infty\}$ for some
quasi-psh function $\f$.
\end{defi}

Pluripolar sets 
are the ``small sets" of pluripotential theory.
In the definition above one can further assume that $\f$ is $\omega$-psh. 
There are many ways to characterize pluripolar sets, we refer the
 reader to the book \cite{GZbook}.

\begin{defi}
Given a Lebesgue measurable function $h:X \rightarrow \R$, its $\omega$-psh envelope is defined by
$$
P(h):= \left ( \sup \left\{ u \;  |  \; u \in \PSH(X,\omega) \text{ and } u \leq h \text{ on X } \right\} \right)^*.
$$
\end{defi}

Here $^*$ denotes the upper semi-continuous regularization. The function $P(h)$ is either identically $-\infty$ or it belongs to $\PSH(X,\omega)$.

\begin{thm} \label{thm:BD}
 \cite{BD12}
If $h$ is smooth then
$P(h)$ has bounded Laplacian, $\omega+dd^c h\geq 0$ on the contact set $\{P(h)=h\}$, and  
$$
(\omega+dd^c P(h))^n=1_{\{P(h)=h\}} (\omega+dd^c h)^n.
$$
\end{thm}

The fact that $\omega+dd^c h\geq 0$ on the contact set $\{P(h)=h\}$ follows from basic properties of plurisubharmonic functions. We refer the reader to \cite[Page 46]{BD12} or \cite[Proof of Proposition 1.3]{EGZ11} for a proof of this fact. 
An alternative proof of this result has been provided by Berman in \cite{Berm19}. 

\subsection{Finite energy classes}


Given   $\f \in \PSH(X,\omega)$, we consider  
$$
\f_j:=\max(\f, -j) \in \PSH(X,\omega) \cap L^{\infty}(X).
$$
It follows from the Bedford-Taylor theory \cite{BT76,BT82} that the measures 
$(\omega+dd^c \f_j)^n
$
are well defined probability measures and the sequence
$$
\mu_j:={\bf 1}_{\{ \f>-j\}} (\omega+dd^c \f_j)^n
$$
is increasing \cite[p.445]{GZ07}. Since the $\mu_j$'s all have total mass bounded from above by $1$,
we consider
$$
\mu_{\f}:=\lim_{j \rightarrow +\infty} \mu_j,
$$
which is a positive Borel measure on $X$, with total mass $\leq 1$.

\begin{defi}
We set 
$$
{\mathcal E}(X,\omega):=\left\{ \f \in \PSH(X,\omega) \; | \; \mu_{\f}(X)=1 \right\}.
$$
For $\f \in {\mathcal E}(X,\omega)$, we set $\omega_{\varphi}^n = (\omega+dd^c \varphi)^n:=\mu_{\f}$.  
\end{defi}

Every bounded $\omega$-psh function clearly belongs to ${\mathcal E}(X,\omega)$.
The class ${\mathcal E}(X,\omega)$ also contains many $\omega$-psh functions which are unbounded:
\begin{itemize}
\item when $X$ is a compact Riemann surface, ${\mathcal E}(X,\omega)$ is precisely the set of $\omega$-sh functions
whose Laplacian does not charge polar sets.
\item if $\f \in \PSH(X,\omega)$ is normalized so that $\f \leq -1$, then $-(-\f)^\e $ belongs to
${\mathcal E}(X,\omega)$ whenever $0 \leq \e <1$.
\item the functions in ${\mathcal E}(X,\omega)$
have relatively mild singularities; in particular they have zero Lelong number at every point. 
\end{itemize}

It is proved in \cite{GZ07}
that the complex Monge-Amp\`ere operator 
$\f \mapsto \omega_{\varphi}^n$ is well defined on the class
${\mathcal E}(X,\omega)$, in the sense that if $\f \in {\mathcal E}(X,\omega)$ then for every sequence of bounded  $\omega$-psh functions $\f_j$ decreasing to $\f$, the  measures $(\omega+dd^c \f_j)^n$ converge weakly on $X$ towards $\mu_\f$.

It follows from \cite{GZ07,DiwJFA09} that a probability measure $\mu$ does not charge pluripolar sets if and only if
there exists a unique $\f \in {\mathcal E}(X,\omega)$ such that
$\mu=(\omega+dd^c \f)^n$ with $\sup_X \f=0$.

\smallskip

When $\mu$ is absolutely continuous with respect to the Lebesgue measure, one expects $\f$ to belong to a 
weighted finite energy class: we
let ${\mathcal W}$ denote the set of all functions $\chi:\R^- \rightarrow \R^-$ such that $\chi$ is increasing and $\chi(-\infty)=-\infty$.

\begin{defi}
Fix $\chi \in {\mathcal W}$.
We let
${\mathcal E}_{\chi}(X,\omega)$ be the set of
$\omega$-psh
functions with
finite $\chi$-energy, 
$$
{\mathcal E}_{\chi}(X,\omega):=\left\{ \f \in {\mathcal E}(X,\omega) \; | \; \chi(-|\f|) \in L^1(X, \MA(\f)) \right\}.
$$
When $\chi(t)=-(-t)^p$, $p>0$, we set ${\mathcal E}^p(X,\omega)={\mathcal E}_\chi(X,\omega)$. 
\end{defi}
For $u\in \mathcal{E}^p(X,\omega)$ the $E_p$ energy is defined as: 
\[
E_p(u) = \int_X |u|^p (\omega +dd^cu)^n.
\]
It follows from \cite[Theorem C]{GZ07}  that a probability measure $\mu$ is the Monge-Amp\`ere measure of a potential in $\mathcal{E}^p(X,\omega)$ if and only if $\mathcal{E}^p(X,\omega) \subset L^p(X,\mu)$. 

The class $\mathcal{E}^p(X,\omega)$, $p\geq 1$, can be equipped with a Finsler metric $d_p$ making it a complete geodesic metric space. We refer the reader to \cite{Dar15} for more details.  We stress here that the distance $d_p$ can be uniformly estimated by pluripotential terms (\cite[Theorem 5.5]{Dar15}): 
\begin{equation}\label{comparison}
C^{-1}d_p(u,v) \leq \int_X |u-v|^p (\omega_u^n+ \omega_v)^n \leq C d_p(u,v), \ u,v\in \mathcal{E}^p(X,\omega),
\end{equation}
for a constant $C=C(p,n)>0$. 

 \subsection{Examples} \label{sec:examples}
 In this section we 
 present some explicit computations suggesting that $\omega$-psh functions 
 with finite entropy   Monge-Amp\`ere measure 
 belong to 
 an
 energy class ${\mathcal E}^p(X,\omega)$ for some $p>1$ 
 depending on $n$.
 
 We recall that the measure $(\omega+dd^c\varphi)^n=f\omega^n$ has finite entropy iff $$\int_X f\log f\; \omega^n <+\infty.$$
 
   \subsubsection{Radial singularities} 
   
   We assume here that the $\omega$-psh functions $\f$ under consideration are locally bounded 
   in $X \setminus \{ p \}$. We choose a local chart biholomorphic to the unit ball $\B$ of $\C^n$,
   with $p$ corresponding to the origin. We further assume that $\f$
   has a {\it radial singularity} at $p$, i.e. it is invariant under the group  
  $O(2n,\R)$ near $p$, and so is  $\omega$.
   The  
   singularity type  of   $\f$ 
    only depends on its local behavior near $p$.
   
 We thus consider the local situation of 
  a psh function $v=\chi \circ L$, where $\chi:\R^- \rightarrow\R^-$ is a 
  convex increasing function and $L:z \in  \B \rightarrow \log |z| \in \R^-$.
  Since we only consider Monge-Amp\`ere measures $(dd^c v)^n$ which are absolutely continuous with respect to Lebesgue measure
  (having finite entropy), we can assume w.l.o.g. that $\chi$ is ${\mathcal C}^1$-smooth
  and $\chi'(-\infty)=0$. Then a standard computation shows that
  $$
  (dd^c v)^n=c_n \frac{(\chi' \circ L)^{n-1}\chi'' \circ L}{|z|^{2n}} dV.
  $$
  We infer that  $(dd^c v)^n $ has finite entropy if and only if
 \begin{equation*}
   \int_{-\infty} (-t) \chi'(t)^{n-1} \chi''(t) dt <+\infty,
 \end{equation*}
while  $v$ belongs to the finite energy class ${\mathcal E}^p(X,\omega)$ if and only if
   \begin{equation*}
   \int_{-\infty} (-\chi(t))^p \chi'(t)^{n-1} \chi''(t) dt <+\infty.
 \end{equation*}
Integrating by parts, 
the finite entropy condition is thus equivalent to
 $$
 \lim_{-\infty} (-t \chi'(t)^n)+ \int_{-\infty}  \chi'(t)^{n}  dt <+\infty.
 $$
 The latter condition implies (since $\chi'$ increases) that $0 \leq \chi'(t) \leq C (-t)^{-1/n}$,
 hence $(dd^c v)^n $ has finite entropy iff $\int_{-\infty}  \chi'(t)^{n}  dt <+\infty$.
Note that 
$$
0 \leq \chi'(t) \leq C (-t)^{-1/n}
\Longrightarrow 0 \leq |\chi(t)| \leq C' |t|^{1-1/n},
$$
thus if $(dd^c v)^n $ has finite entropy then
$$
   \int_{-\infty} (-\chi(t))^p \chi'(t)^{n-1} \chi''(t) dt 
\leq   C'  \int_{-\infty} |t|^{p(1-1/n)} \chi'(t)^{n-1} \chi''(t) dt   <+\infty
$$
  if $p(1-1/n) \leq 1$, i.e. precisely when $p \leq n/(n-1)$.

  \subsubsection{Explicit radial examples} 

We let $X=\C\PP^n$ denote the complex projective space, $n \geq 2$, and $\omega=\omega_{FS}$ be the Fubini-Study K\"ahler form.
We let $\C^n \subset X$ denote an affine chart and
fix $0<\al<1$. Consider 
$$
\chi_{\al}(t)=\left\{
\begin{array}{lcl}
- (-t)^{\al}/\al  & \text{ for } & t \leq -1 \\
t & \text{ for } & t \geq 1
\end{array}
\right. 
$$
and extend $\chi_{\alpha}$ smoothly in $[-1,1]$, so that the resulting  weight $\chi_{\al}$  is convex increasing. Hence
$
\p_{\al}(z)=\chi_{\al} \circ \log |z|
$
is plurisubharmonic in $\C^n$ and extends at infinity as a $\omega$-psh function $\f_{\al}$
s.t.
$$
(\omega+dd^c \f_{\al})^n=(dd^c \p_{\al})^n=f_{\al}\,dV,
$$
 where 
the density $f_{\al}$  
is concentrated in  the euclidean ball $\{|z| \leq e \}$. On $\{e^{-1}<|z|\leq e\}$ it is smooth while on $\{|z|<e^{-1}\}$  it can be written as
$$
f_{\al}(z)=\frac{c'}{(-\log |z| )^{n(1-\al)+1} |z|^{2n}}.
$$
 
Using the general computations above one can check that the entropy\\ $\int_{\{|z| \leq e^{-1} \}} f_{\al} \log f_{\al}\,dV$ is finite if and only if
 $\al < \frac{n-1}{n}$. On the other hand, $\f_{\al}$ belongs 
 to ${\mathcal E}^p(X,\omega)$ if and only if
 $$
 \int_{-\infty} |\chi_{\al}|^p \chi_{\al}'' {\chi'}_{\al}^{n-1}  \, dV
 \simeq \int_{-\infty} \frac{|t|^{p\al}}{|t|^{n(1-\al)+1}}< +\infty
 \Longleftrightarrow
 p< n \frac{1-\al}{\al}.
 $$
 
Fix $\e>0$ very small. 
If we choose $\al=\frac{n-1}{n+\e}$ then the probability measure
$(\omega+dd^c \f_{\al})^n$ has finite entropy and
$\f_{\al}$ belongs to the finite energy class ${\mathcal E}^{\frac{n}{n-1}}(X,\omega)$,
but it does not belong to ${\mathcal E}^{\frac{n}{n-1}(1+\e)}(X,\omega)$.


    \subsubsection{Divisorial singularities}\label{div Section}

We assume here $X$ is projective, $\omega=i \Theta_h$ is the (positive) curvature of 
a smooth hermitian metric $h$ of an ample holomorphic line bundle $L$ on $X$
(i.e. $\omega$ is a Hodge form), and $D=(s=0)$ is a divisor defined as the zero set of a holomorphic section $s \in H^0(X,L)$.

We assume w.l.o.g. that the $\omega$-psh function $\p=\log |s|_h$
is normalized by $\sup_X \p \leq -1$. By construction 
$\omega+dd^c \p=[D]$ is the current of integration along $D$. Let $\chi: \R^- \rightarrow \R^-$
denote a smooth convex increasing function such that $\chi' \leq 1/2$. The function $\f:=\chi \circ \p$ then
satisfies
$$
\omega_\f:=\omega+dd^c \f=[1-\chi' \circ \p] \omega +\chi' \circ \p\, \omega_{\p}+\chi'' \circ \p\, d\p \wedge d^c \p \geq 0
$$
It is thus  $\omega$-psh and has {\it divisorial singularities} along $D$.

We assume $\chi'(-\infty)=0$. Thus $\chi' \circ \p \omega_{\p}=\chi' \circ \p [D]=0$ and 
$$
\omega_\f=[1-\chi' \circ \p] \omega +\chi'' \circ \p d\p \wedge d^c \p \geq 0.
$$
We let the reader then check that
$$
(\omega+dd^c \f)^n=f \omega^n
\text{ with }
f \sim 1+\frac{\chi'' \circ \p}{|s|^2}.
$$
 
The finite entropy condition near a smooth point of $D=(z_1=0)$ reads as
$$
\int_0 \frac{(-\log|z_1| ) \chi'' \circ \log |z_1|}{|z_1|^2} dV(z_1) =\int_{-\infty} (-t) \chi''(t) dt <+\infty,
$$
passing to (logarithmic) polar coordinates. An integration by parts yields
\begin{eqnarray*}
\int_{-\infty} (-t) \chi''(t) dt &= &\lim_{t \rightarrow -\infty} (-t) \chi'(t) +\int_{-\infty} \chi'(t) dt \\
&= & \lim_{t \rightarrow -\infty} (-t) \chi'(t)+(-\chi)(-\infty)+O(1).
\end{eqnarray*}

This shows that $\f=\chi \circ \log |s|_h$ is a finite entropy potential if and only if $\chi(-\infty)>-\infty$,
i.e. when $\f$ is bounded. Said differently, there is no finite entropy potential with divisorial singularities.

\section{Moser-Trudinger inequalities}\label{sect: MT inequality}

  \begin{thm}
 	\label{thm: Moser-Trudinger inequality} 
 	Fix $p >0$. Then there exist $c>0, C>0$ depending on $X$, $\omega$, $n$, $p$ such that, for all $\varphi\in \mathcal{E}^p(X,\omega)$ with $\sup_X \varphi=-1$,  we have
 	\[
 	\int_X \exp \left ( c  |E_p(\varphi)|^{-\frac{1}{n}} |\varphi|^{1+\frac{p}{n}}\right )\omega^n \leq C. 
 	\]
 \end{thm}

  Theorem \ref{thm: Moser-Trudinger inequality} generalizes an important result of Berman-Berndtsson \cite[Theorem 1.1]{BB11} who established the above inequality in the case when $\{\omega\}=c_1(L)$ for some holomorphic line bundle $L$ and 
  $p=1$. 
  
 \begin{proof}
 	By approximation we can assume that $\varphi$ is smooth. For notational convenience we set $q=\frac{n+p}{n}>1$, $\psi:=-a(-\varphi)^q$ and $u := P(\psi)$, where $a>0$ is a small constant suitably chosen below. 
 	Note also that $E_p(\varphi) = \int_X(-\varphi)^p \omega_{\varphi}^n \geq 1$ because $\varphi\leq -1$.  A direct computation shows that  
 	 \begin{eqnarray*}
\omega+dd^c \psi & =& \omega+aq(-\varphi)^{q-1} dd^c \varphi -aq(q-1) (-\varphi)^{q-2} d\varphi \wedge d^c \varphi\\
  &\leq & [1-aq(-\varphi)^{q-1}] \omega+ aq(-\varphi)^{q-1} (\omega+dd^c \varphi).
	 \end{eqnarray*}
We set 
\[
F:= \left\{aq|\varphi|^{q-1} \geq 1\right\}=\left\{qa^{\frac{n}{n +p}} |\psi|^{\frac{p}{n+p}}\geq 1\right\},\, \ G:= \left\{|u| \geq a^{-\frac{n}{p}} q^{-\frac{n+p}{p}}\right\}.
\] 
Observe that 
$$
(\omega+dd^c \psi)\leq aq|\psi|^{q-1} (\omega+dd^c \varphi)\quad  {\rm{on}}\; F
$$ 
and $G \cap \{\psi=P(\psi)\} \subseteq F$. Hence on $G\cap \{\psi=P(\psi)\}$ we have 
\[
0\leq  (\omega +dd^c \psi) \leq aq|\varphi|^{q-1} (\omega+dd^c \varphi),
\]
where the first inequality follows from
Theorem \ref{thm:BD}.
By Theorem \ref{thm:BD} again,
\begin{eqnarray}\label{mass1}
{\bf 1}_{F}(\omega+dd^c u)^n &= &1_{\{ \psi=P(\psi)\} \cap F}  (\omega+dd^c \psi)^n \nonumber \\
&\leq & 1_{\{ \psi=P(\psi)\} \cap F} \,a^n q^n |\psi|^{p} (\omega+dd^c \psi)^n.
\end{eqnarray}
We now choose $a\in(0,1)$ so that 
\[
2a^n q^n \int_X |\varphi|^p (\omega+dd^c \varphi)^n=2a^n q^n E_p(\varphi) = 1.
\]
Integrating over $X$ both sides of \eqref{mass1}, we obtain
\[
\int_G (\omega+dd^c u)^n \leq \frac{1}{2}.
\]
It thus follows that $G\neq X$, or equivalently that $G^c\neq \emptyset$. In particular $\sup_X u \geq -a^{-\frac{n}{p}} q^{-\frac{n+p}{p}}=:-b$. Recall that we have chosen $c_0,C_0>0$ so that, for all $\phi \in \PSH(X,\omega)$ with $\sup_X \phi=0$ we have 
\[
\int_X e^{c_0|\phi|} \omega^n \leq C_0.
\]
In particular, for $\phi= u+b$ we have $\sup_X \phi \geq 0$, hence
\[
\int_X e^{c_0 (a (-\varphi)^q -b)} \omega^n \leq  \int_X e^{-c_0 (u +b)} \omega^n = \int_X e^{-c_0 (\phi-\sup_X \phi) - c_0\sup_X \phi}  \leq C_0. 
\]
Set $K=\{a|\varphi|^{q-1} \leq q^{-1} 2^{1-1/q}\}$ and note that $a|\varphi|^q \leq q^{-1}2^{1-1/q}|\varphi|\leq 2|\varphi|$ on $K$.
On $X\setminus K$ we have that $a^{\frac{q}{q-1}} |\varphi|^q > 2q^{-\frac{q}{q-1}}$ or, equivalently,
\[
\frac{a|\varphi|^{q}}{2} <  a|\varphi|^q -  b.
\]
Thus
\begin{eqnarray*}
\int_X e^{c_0\frac{ a |\varphi|^q}{2}} \omega^n &\leq &\int_K e^{c_0 |\varphi|} \omega^n + \int_{X\setminus K}e^{c_0(a|\varphi|^q -b)} \omega^n\\
&=&\int_K e^{-c_0 (\varphi+1)+c_0} \omega^n + \int_{X\setminus K}e^{c_0(a|\varphi|^q -b)} \omega^n 
\leq C_0(e^{c_0}+1). \\
\end{eqnarray*}
The result follows with 
$
c= 2^{-1-1/n}q^{-1}c_0,\, {\rm{and}} \, C= C_0(e^{c_0}+1). 
$
\end{proof}

An interesting consequence of Theorem \ref{thm: Moser-Trudinger inequality} 
is the following result 
which was conjectured by Aubin \cite[page 148]{Aub84} when $p=1$:

\begin{corollary}
	\label{cor: MT inequality}
	Fix $p>0$. There exist uniform constants $A,B>0$ such that for all $k\in \mathbb{N}$ and $\varphi\in \mathcal{E}^p(X, \omega)$ normalized by $\sup_X \varphi=-1$, 
	\[
	\log \int_X e^{- k\varphi} \omega^n  \leq A k^{1+\frac{n}{p}} E_p(\varphi)^{1/p} +B. 
	\]
\end{corollary}

We follow \cite[page 359]{BB14}.  As explained in \cite[page 2]{BB11} Corollary \ref{cor: MT inequality} is actually equivalent to 
Theorem \ref{thm: Moser-Trudinger inequality}.  A similar estimate has been observed in \cite[Remark page 1461]{AC19} 
for the case when $p\geq 1$ (and $\{\omega\}$ is integral) by a direct application of \cite{BB11} and H\"older inequality.

\begin{proof}
	We set $\alpha = 1+ \frac{p}{n}$, $\beta = 1+ \frac{n}{p}$ so that $\frac{1}{\alpha} + \frac{1}{\beta}=1$. For $\eta>0$, $\xi >0$, Young's inequality yields
	\[
	\eta\xi \leq \frac{\eta^{\beta}}{\beta} + \frac{\xi^{\alpha}}{\alpha} = \frac{p \eta^{1+\frac{n}{p} }}{n+p} + \frac{n \xi^{1+\frac{p}{n}}}{n+p}.
	\]
	Dividing by $\eta$ and replacing $\eta$ by $\tau^{1+\frac{p}{n}}$
	we obtain 
	\[
	\xi \leq \frac{p \tau^{1+\frac{n}{p}}}{n+p} + \frac{n \xi^{1+ \frac{p}{n}}}{(n+p)\tau^{1+ \frac{p}{n}}} .
	\]
	We now choose $\xi = -k\varphi(x)$, $x\in X$, and $\tau= a E_p(\varphi)^{\frac{1}{p+n}}$, $a>0$, to obtain 
	\begin{equation}\label{ineq3.2}
 	-k\varphi(x)  \leq  \frac{p a^{1+\frac{n}{p}} E_p(\varphi)^{\frac{1}{p}} }{n+p} +  \frac{n (-k\varphi(x))^{1+\frac{p}{n}}}{(n+p) a^{1+\frac{p}{n}} E_p(\varphi)^{\frac{1}{n}}}.
	\end{equation}
	We   choose $a$ so that 
	$
	\frac{n k^{1+\frac{p}{n}}}{(n+p)a^{1+\frac{p}{n}}} = c,
	$
	where $c$ is the constant in Theorem \ref{thm: Moser-Trudinger inequality}, i.e.
	$
	a= k  \left (\frac{n}{(n+p)c} \right )^{\frac{n}{n+p}}. 
	$
	By inserting this choice of $a$ in \eqref{ineq3.2} we get
	\begin{equation*}	
	-k\varphi(x)  \leq  \frac{p}{(n+p) }\left(\frac{n}{(n+p)c}\right)^{n/p} k^{1+\frac{n}{p}} E_p(\varphi)^{\frac{1}{p}} + c|\varphi(x)|^{1+\frac{p}{n}}E_p(\varphi)^{-\frac{1}{n}}.
	\end{equation*}
	Exponentiating and integrating, we thus obtain the desired inequality with 
	\[
	A=  \frac{p n^{n/p}}{c^{n/p} (n+p)^{1+n/p}}, \;
	\text{ and } \;  B= \log C, 
	\]
	where $C$ is the constant in Theorem \ref{thm: Moser-Trudinger inequality}. 
\end{proof}

We now connect integrability properties of a given qpsh function $v$  
to the finite energy of its envelopes $P(-(-v)^q)$,  $q>1$. 
Since $q>1$, the function $-(-v)^q$ is not qpsh and  it is
more singular than $v$. For an arbitrary  qpsh function $v$,
 it could happen that $P(-(-v)^q)$ is identically $-\infty$.

 \begin{prop}\label{lem: criterion in Ep}
Fix $v\in \PSH(X,\omega)$ with $v\leq -1$ and $p>0$.  
\begin{itemize}
\item  If $v\in \mathcal{E}^p(X,\omega)$ then $P(-(-v)^{1+ \frac{p}{n}}) \in \mathcal{E}(X,\omega)$. 
\item If $P(-(-v)^{s}) \in \PSH(X,\omega)$ for some $s>1+p$ then $v\in \mathcal{E}^p(X,\omega)$. 
\end{itemize}
 \end{prop}

 \begin{proof}
Assume that $v \in \mathcal{E}^p(X,\omega)$.  By approximation we can assume that $v$ is smooth. 
Set $q=1+ \frac{p}{n}$, $u=-(-v)^q$, and $\psi = P(u)$. We compute
 \begin{eqnarray*}
  \omega+dd^c u & =& \omega+q(-v)^{q-1} dd^c v -q(q-1) (-v)^{q-2} dv \wedge d^c v\\
  &\leq & [1-q(-v)^{q-1}] \omega+q(-v)^{q-1} (\omega+dd^c v) \leq  q(-v)^{q-1} (\omega+dd^c v) .
 \end{eqnarray*}
It then follows from Theorem \ref{thm:BD} that
$$
(\omega+dd^c \p)^n =1_{\{ u=P(u)\}} (\omega+dd^c u)^n  
= 1_{\{ u=P(u)\}} q^n |v|^{n(q-1)} (\omega+dd^c v)^n.
$$
Since $v\in \mathcal{E}^p(X,\omega)$, it follows from \cite{GZ07} 
that there exists  $C>0$ and an increasing function $\chi_1 : \mathbb{R}^- \rightarrow \mathbb{R}^{-}$ 
with $\chi_1(-\infty)=-\infty$, such that  for all $j$,
 	\[
 	\int_X |\chi_1\circ v| |v|^p (\omega+dd^c v)^n \leq C.
 	\]
 	
Consider $\chi_2:\mathbb{R}^- \rightarrow \mathbb{R}^{-}$ defined such that $\chi_1(t) = \chi_2 (-(-t)^{q}))$. Observe that $\chi_2$ is again an increasing function and $\chi_2(-\infty)=-\infty$. 
Now
\begin{eqnarray*}
 \int_X |\chi_2 \circ \psi | (\omega+dd^c \p)^n  &=&  \int_{\{u=P(u)\}} |\chi_2 \circ u| (\omega+dd^c u)^n  \\
 & \leq & q^n \int_X | \chi_2 \circ u |  (-v)^{n(q-1)} (\omega+dd^c v)^n \\
 &= & q^n \int_X |\chi_1 \circ v |  |v|^p (\omega+dd^c v)^n  \leq  Cq^n.
 \end{eqnarray*}
This shows that $\p\in{\mathcal E}_{\chi_2}(X,\omega)$ and in particular $\p\in{\mathcal E}(X,\omega)$.

For the reverse implication we assume that $\p:=P(-(-v)^s) \in \PSH(X,\omega)$  for some $s>1+p$ and set $w:= -(-\p)^{1/s}$.
 It follows from Lemma \ref{lem: Ep GZ07} below that $w\in \mathcal{E}^p(X,\omega)$,
 hence $v$ also belongs to $\mathcal{E}^p(X,\omega)$ since $v \geq w$.
\end{proof}

\begin{lemma}\label{lem: Ep GZ07}
Fix $p>0$ and $r<\frac{1}{1+p}$. If $u\in \PSH(X,\omega)$ and $u\leq -1$ then $-(-u)^r \in \mathcal{E}^{p}(X,\omega)$. 
\end{lemma}

The exponent $r$ is sharp: the function  $\varphi:=-(-\log|s|)^{r}$ (introduced in Section \ref{div Section}) 
 belongs to $ \mathcal{E}^{1}(X,\omega)$ iff $r<1/2$ \cite[Proposition 2.8]{Dn15}.

\begin{proof}
    Let $(u_j)$ be a decreasing sequence of smooth $\omega$-psh functions such that $u_j\searrow u$
and $u_j\leq -1$. Set $w_j:= -(-u_j)^r$.  A direct computation shows 
 \begin{eqnarray*}
\omega +dd^c w_j &=& \omega+ r |u_j|^{r-1} dd^c u_j+ r(1-r) |u_j|^{r-2} d u_j \wedge d^c u_j\\
 &=& \left(1-r |u_j|^{r-1}\right)\omega+ r |u_j|^{r-1} (\omega+dd^c u_j) \\
 &&+ r(1-r) |u_j|^{r-2} du_j \wedge d^c u_j\\
 & \leq & \left(1-r |u_j|^{r-1}\right)\omega+ |u_j|^{r-1} (\omega+dd^c u_j)+ |u_j|^{r-2} du_j \wedge d^c u_j,
 \end{eqnarray*}
using that $u_j\leq -1$ and $r<1$.
We now integrate on $X$ to obtain
\begin{align*}
\int_X (-w_j)^p & (\omega+dd^c w_j)^n \leq   \int_X |u_j|^{pr} (\omega + 	|u_j|^{r-1} (\omega+dd^c u_j))^n \\
&+ n \int_X  |u_j|^{pr+r-2} (\omega + 	|u_j|^{r-1} (\omega+dd^c u_j))^{n-1} \wedge du_j \wedge d^c u_j. 
\end{align*}
We bound the first term by 
\begin{align*}
	\int_X &  |u_j|^{pr}(\omega+|u_j|^{r-1} (\omega+dd^c u_j))^n \\ 
	& \leq  2^n \sum_{k=0}^n\int_X |u_j|^{pr}  (|u_j|^{r-1}(\omega +dd^c u_j))^{k} \wedge \omega^{n-k} \leq 2^n (n+1),
\end{align*} 
noticing that $r(p+1)-1<0$, hence $|u_j|^{r(p+1)-1}\leq 1$. 
For the second term we write: 
\begin{align*}
	|u_j|^{pr+r-2} & (\omega + 	|u_j|^{r-1}(\omega+ dd^c u_j))^{n-1} \wedge du_j \wedge d^c u_j \\
	& \leq  2^{n-1}|u_j|^{q} du_j \wedge d^c u_j \wedge \sum_{k=0}^{n-1} (\omega+dd^c u_j)^k\wedge \omega^{n-k-1}\\
	& = \frac{2^{n-1}}{(q+1)} d (-(-u_j)^{q+1}) \wedge d^c u_j \wedge \sum_{k=0}^{n-1}(\omega+dd^c u_j)^k\wedge \omega^{n-k-1},
\end{align*}
where $q:=pr+r-2<-1$.  Integrating by parts we obtain
\begin{align*}
	\int_X |u_j|^{pr+r-2} & (\omega + 	|u_j|^{r-1} (\omega+ dd^c u_j))^{n-1} \wedge du_j \wedge d^c u_j \\
	& \leq \frac{2^{n-1}}{(q+1)} \int_X  |u_j|^{q+1} dd^c u_j \wedge \sum_{k=0}^{n-1} (\omega+dd^c u_j)^k\wedge \omega^{n-k-1}\\
	& \leq \frac{2^{n-1}}{(q+1)}  \sum_{k=0}^{n-1} \int_X (\omega+dd^c u_j)^{k+1}\wedge \omega^{n-k-1} = \frac{n 2^{n-1}}{(q+1)}.
\end{align*}
Thus $\int_X (-w_j)^p (\omega+dd^c w_j)^n \leq C(n,p,r)$, proving that $w\in \mathcal{E}^p(X,\omega)$. 
\end{proof}

\section{Finite entropy potentials} 


In this Section we systematically study finite entropy potentials.
For each $B>0$ we set 
\[
\Ent_B := \{u \in \mathcal{E}^1(X,\omega) \; | \;  \sup_X u =-1, \Ent_{\omega^n}(\omega_u^n) \leq B \}.
\]
Recall that ${\rm Ent}_{\omega^n} (\omega_u^n):=\int_X f \log f\; \omega^n$
if  $\omega_u^n=f \omega^n$ is absolutely continuous with respect to $\omega^n$ and $\Ent_{\omega^n}(\omega_u^n)=+\infty$ otherwise.\\ 


\subsection{Continuity of the mean-value} 

Given a probability measure $\mu$ with finite entropy we denote by $L_{\mu}$ the map
\[ u \in \PSH(X,\omega) \mapsto  \int_X u \; d\mu.  \]

We consider the convex function $\chi: s \in \R^+ \mapsto (s+1) \log (s+1) -s \in \R^+$. Its conjugate convex function is
  $$
\chi^*: t \in \R^+ \mapsto \sup_{s>0} \{ st-\chi(s) \}=e^t-t-1 \in \R^+.
  $$
By definition these functions satisfy, for all $s,t>0$,
\begin{equation} \label{eq:conj}
st \leq \chi(s)+\chi^*(t).
\end{equation}
Such simple inequality will be used to prove the following:
\begin{prop}\label{prop: Lmu continuous}
Let $\mu=f\omega^n$ be a probability measure on $X$ with finite entropy $\int_X f\log f \omega^n<+\infty$. Then $L_\mu$ is well-defined and continuous.
\end{prop}
\noindent Here $\PSH(X,\omega)$ is endowed with the $L^1$-topology.

\begin{proof}
Thanks to \cite[Theorem 2.50]{GZbook} we can ensure that there exist $c,C>0$ such that $\int_X e^{4c|\psi|} \omega^n \leq C$ for all $\psi\in \PSH(X,\omega)$ normalized with $\sup_X \psi=0$. Using \eqref{eq:conj} with $s=f(x)$ and $t=-c\psi(x)$, $x\in X$, we have
\[
c|\p| f \leq e^{c|\p|} + \chi(f). 
\]
Integrating over $X$ we see that $L_{\mu}$ is well-defined (since $f$ has finite entropy). We now prove that $L_{\mu}$ is continuous. 
 Assume that $\psi_j$ is a sequence in $\PSH(X,\omega)$ such that $\|\psi_j-\psi\|_{L^1(X)} \to 0$ for some $\psi\in \PSH(X,\omega)$.  Since $\sup_X \psi_j \to \sup_X \psi$, we can assume that $\sup_X \psi_j=\sup_X \psi=0$.  Up to extracting a subsequence, we can also assume that $\psi_j$ converges a.e. to $\psi$. By the choice of $c$ and elementary inequalities we can show that 
 $\|e^{-2c\psi_j}-e^{-2c\psi}\|_{L^1(X)} \to 0$.
 Thus, passing to a subsequence we can assume that $e^{-2c\psi_j} \leq g$ for some $g\in L^1(X)$. 
Applying \eqref{eq:conj} with $s=f(x)$ and $t=c|\psi_j-\psi|(x)$, $x\in X$, we obtain
 \[
 c|\psi_j-\psi| f \leq \chi(f) + e^{c|\p_j-\p|}
 \leq \chi(f) + e^{c|\p_j|}e^{c|\p|}
 \leq \chi(f)+ g+ e^{2c|\psi|}  \in L^1(X).
 \]
The dominated convergence theorem thus yields $\int_X |\psi_j-\psi| f \omega^n \to 0$. 
\end{proof}

\subsection{Compact Riemann surfaces} \label{sect: dimension 1}
In this section we treat the case of compact Riemann surfaces. We thus assume that $n=1$ and the set $\PSH(X,\omega)$ will be denoted by $\SH(X,\omega)$, where the latter is endowed with the $L^1$-topology. The results in this section are probably well known to  experts. We include them as a warm-up for the reader (this is Exercise 12.1, page 339 in \cite{GZbook}). 

\begin{lem} \label{lem: char dim 1}
Let $\mu=\omega+dd^c \f$, $\varphi \in \SH(X,\omega)$ be a probability measure. Then
\begin{enumerate}
\item $\f$ is bounded if and only if ${\rm SH}(X,\omega) \subset L^1(\mu)$;
\item $\varphi$ is continuous if and only if $L_{\mu}$ is continuous on $\SH(X,\omega)$. 
\end{enumerate}
\end{lem}

This characterization does not hold when $n\geq 2$:
if $\mu=(\omega+dd^c \f)^n$
has finite entropy then  $\PSH(X,\omega) \subset L^1(\mu)$ and $L_{\mu}$ is continuous on $\PSH(X,\omega)$  
but this does not necessarily imply that $\f$ is bounded 
(see Section \ref{sec:examples}).

\begin{proof} We prove (1). 
  Assume first that $\f$ is bounded and fix $\p \in {\rm SH}(X,\omega)$.
We want to show that $\int_X |\p| \, d\mu <+\infty$. Since $\p$ is bounded from above, we can assume without
loss of generality that $\p \leq 0$. Since $\f$ is bounded from below, we can assume $\f \geq 0$. Now by Stokes theorem,
$$
\int (-\p) \, d\mu=\int (-\p) \, \omega+\int  \f  (-dd^c \p)
\leq \int (\f-\p) \, \omega <+\infty,
$$
since $\f \geq 0$ and $-dd^c \p \leq \omega$. This shows ${\rm SH}(X,\omega) \subset L^1(\mu)$.

\smallskip

Assume now that ${\rm SH}(X,\omega) \subset L^1(\mu)$.
We need to show that $\f$ is bounded  from below.
Let $\delta_a=\omega+dd^c \psi_a$, with $\p_a \in {\rm SH}(X,\omega)$ and $\sup_X \p_a =0$, be the Dirac mass at a point $a \in X$. 
By \cite[Proposition 2.7]{GZ05}, there there exists $C_{\mu}>0$ such that
$$
\forall a \in X, \;  \int_X \p_a d\mu   \geq -C_{\mu}.
$$
Integrating by parts we therefore obtain that for all $a \in X$,
$$
\f(a)=\int_X \f \, \delta_a=\int_X (\f -\p_a) \omega+\int_X \p_a \, d\mu
\geq -C_{\mu}+\int_X \f \omega,
$$
hence $\f$ is bounded.

We next prove (2). Let $(a_j) \in X^{\N}$ be a sequence of points converging to $a \in X$. With the same notations as above,
it follows from Stokes theorem that
$$
\int_X \p_{a_j} \, d\mu -\int_X \p_a \, d\mu=\int_X (\p_{a_j}-\p_a) \, \omega+[\f(a_j)-\f(a)].
$$
We claim that $\psi_{a_j}$ converges to $\p_a$ in $L^1$. Indeed, by construction we have that $\delta_{a_j}=\omega+dd^c \psi_{a_j}$ weakly converges to $\delta_a=\omega+dd^c \psi_{a}$. On the other hand by compactness \cite[Proposition 2.7]{GZ05}, $\p_{a_j}$ converges in $L^1$ to a potential $\tilde{\psi}_a$ solving $\delta_a=\omega+dd^c \tilde{\psi}_{a}$. By uniqueness we get $\tilde{\psi}_a=\p_a$. This proves the claim. Hence, if the mapping $u\in \SH(X,\omega) \mapsto \int_X u \, d \mu \in \R$ is continuous, so is $\f$.

Conversely assume $\f$ is continuous. Let $(\p_j)$ be a sequence of $\omega$-psh functions that converges
in $L^1(X)$ to $\p \in \SH(X,\omega)$. Then 
$$
\int_X \p_j \, d\mu -\int_X \p \, d\mu=\int_X (\p_j-\p) \, \omega+ \left[ \int_X \f (\omega+dd^c \p_j)-\int_X \f (\omega+dd^c \p) \right].
$$
Now $\omega+dd^c \p_j$ converges to $\omega+dd^c \p$ in the sense of distributions hence in the sense
of Radon measures, since these are probability measures. We infer that
$\int_X \f (\omega+dd^c \p_j) \rightarrow \int_X \f (\omega+dd^c \p)$ hence $\p \mapsto \int \p \, d\mu$ is continuous.
\end{proof}

 As a direct consequence of Lemma \ref{lem: char dim 1} and Proposition \ref{prop: Lmu continuous} we obtain:
\begin{coro}
If $\mu=\omega+dd^c \varphi$ has finite entropy then $\varphi$ is continuous. 
\end{coro}

It is also worth to mention that Example \ref{ex: inj_nocompact} below shows that the injection $\Ent_B(X,\omega) \hookrightarrow \mathcal{C}^0(X)$, when $n=1$, is not compact.

\subsection{Higher dimensional compact setting}\label{sect: higher dimension}

\begin{theorem}
	\label{thm: Entropy implies Ep and integrability} 
	Fix $B>0$ and set $p=\frac{n}{n-1}$. 
	There exist $c,C>0$ depending on $B,X,\omega,n$ such that, for all $\varphi\in \Ent_B$ we have 
	\[
	\int_X e^{c (-\varphi)^p} \omega^n \leq C\, \ \text{and}\ E_p(\varphi) \leq C. 
	\]
	In particular, $\Ent(X,\omega) \subset \mathcal{E}^{\frac{n}{n-1}}(X,\omega)$. 
 \end{theorem}

\begin{proof}
	Fix $B>0$ and $\varphi \in \Ent_B$ with $\omega_{\varphi}^n =f\omega^n$.

%
  \smallskip
  \noindent    {\it Step 1: the function $\f$ belongs to 
  $ {\mathcal E}^p(X,\omega)$.}  \\
We claim that $\mathcal{E}^p(X,\omega) \subset L^p(X,\mu)$. Indeed, fix $v \in \mathcal{E}^p(X,\omega)$ with $\sup_X v=-1$ and observe that $1+\frac{p}{n} = p$. It follows from Theorem \ref{thm: Moser-Trudinger inequality} that for some $c>0$ small enough, 
\begin{equation}\label{unif_bound_exp}
\int_X e^{c |v|^p} \omega^n <+\infty.
\end{equation} 

We apply \eqref{eq:conj} with $s=f(x)$ and $t=c|v(x)|$.
This yields
\[
\int_X c|v|^p \omega_\f^n= \int_X c |v|^p f \omega^n \leq \int_X \chi\circ f \omega^n +\int_X (e^{c |v|^p}-c |v|^p-1) \omega^n,
\]
where the first integral is finite since $f$ has finite entropy,
while the second is bounded thanks to \eqref{unif_bound_exp} and the integrability properties of qpsh functions. 
This means that $\int_X |v|^p \omega_\f^n<+\infty$, proving the claim. By \cite[Theorem C]{GZ07} we can thus infer that $\f\in \mathcal{E}^p(X,\omega)$. 
\smallskip

\noindent  {\it Step 2: Integrability of $ e^{ |\f|^p}$ and energy bound.} \\ 
Using the H\"older-Young inequality  again we see that 
 \begin{eqnarray*}
  	 \int_X c \frac{|\f|^{p}}{ E_p(\f)^{1/n}}   f \omega^n 
  	 & \leq & \int_X \left(e^{c \frac{|\f|^{p}}{ E_p(\f)^{1/n}}} -c \frac{|\f|^{p}}{ E_p(\f)^{1/n}}-1 \right)\omega^n 
  	 + \int_X \chi\circ f \omega^n\\
  	 &\leq & C_1.
  \end{eqnarray*}
 Theorem \ref{thm: Moser-Trudinger inequality} insures that $C_1>0$ depends only on $X$, $\omega, n,$ and $\Ent(f)$. 
From the above inequality we get $c E_p(\f)^{1-1/n} \leq C_1$,
 which  yields
 $E_p(\f) \leq C_2$.
 Invoking Theorem \ref{thm: Moser-Trudinger inequality} again we obtain 
  \[
  \int_X e^{\gamma |\f|^p} \omega^n \leq C_3,
  \]
where $\gamma=c\, C_2^{-1/n}>0$ is a uniform constant.
 \end{proof}

The examples from Section \ref{sec:examples} show that the exponent $\frac{n}{n-1}$ is sharp.

\begin{theorem}
	\label{thm: compact injection}
	The set $\Ent_B$ is compact in $(\mathcal{E}^p(X,\omega),d_p)$, for any $p<\frac{n}{n-1}$. 
\end{theorem}  
 
\begin{proof} 
Fix $p<r:=\frac{n}{n-1}$, and $B>0$. 
 Let $\varphi_j$ be a sequence in $\Ent_B$.  
 By \cite[Theorem 2.17]{BBEGZ19}  and \eqref{comparison} 
 there exists a subsequence of $\varphi_j$, still denoted by $\varphi_j$, such that $d_1(\varphi_j,\varphi) \to 0$ as $j\to +\infty$ for some $\varphi \in \mathcal{E}^1(X,\omega)$ with $\sup_X \varphi=-1$. If we write $g_j\omega^n = (\omega+dd^c \varphi_j)^n + (\omega+dd^c \varphi)^n$ then  
 \[
 \int_X |\varphi_j-\varphi| g_j \omega^n \to 0.
 \]
 By lower semicontinuity of the entropy \cite[Proposition 2.10]{BBEGZ19} we also have that ${\rm Ent}_{\omega^n}(\omega_{\varphi}^n) \leq B$, hence by Theorem \ref{thm: Entropy implies Ep and integrability}, $\varphi\in \mathcal{E}^{\frac{n}{n-1}}(X,\omega)$.  
We want to prove that $d_p(\varphi_j,\varphi)\to 0$ which by \eqref{comparison} is equivalent to showing that 
 \[
 \lim_{j\to +\infty} \int_X   |\varphi_j-\varphi|^p g_j \omega^n =0.
 \] 
 
 Fix $r\in (0,1)$ such that $(1-r)\frac{n}{n-1} +r=p$, that is $p=\frac{n-r}{n-1}$. Using the H\"older inequality with exponent $\frac{1}{r}$ and $\frac{1}{1-r}$ we obtain 
 \begin{align}\label{compactness ineq}
     \int_X |\varphi_j-\varphi|^p g_j \omega^n & = \int_X |\varphi_j-\varphi|^{r+(1-r)\frac{n}{n-1}} g_j \omega^n \nonumber \\
     & \leq  \left (\int_X |\varphi_j-\varphi| g_j \omega^n \right )^r  \left (\int_X |\varphi_j-\varphi|^{\frac{n}{n-1}} g_j \omega^n\right )^{1-r}.
 \end{align}
Now, $|\f_j-\f|^{\frac{n}{n-1}}\leq C_1 (|\f_j|^{\frac{n}{n-1}}+|\f|^{\frac{n}{n-1}})$ and $g_j \omega^n$ has finite entropy. 
By H\"older-Young inequality and Theorem \ref{thm: Entropy implies Ep and integrability} we   infer  
\[
\int_X c|\f_j|^{\frac{n}{n-1}}g_j\omega^n\leq \int_X \chi\circ g_j \omega^n +\int_X (e^{c |\f_j|^{\frac{n}{n-1}}}-c |\f_j|^{\frac{n}{n-1}}-1) \omega^n\leq C.
\]
Similarly we have $\int_X |\f|^{\frac{n}{n-1}}g_j\omega^n\leq C$. 
The second factor  in \eqref{compactness ineq} is thus bounded while the first one converges to $0$ as $j\to +\infty$. 
\end{proof}

The injection $\Ent(X,\omega) \hookrightarrow \mathcal{E}^{\frac{n}{n-1}}(X,\omega)$ is however not compact:
 
\begin{exa}\label{ex: inj_nocompact}
Let $\chi: \R \rightarrow \R$ be a smooth increasing convex function such that
$\chi(t) =0 \text{ for } t \leq -\log 2$, $\chi(t) =  t$ for $t \geq \log 2$, and $\chi$ is strictly convex
in $(-\log 2, \log 2)$.  
We fix $\e_j>0$ a sequence decreasing to zero, $C_j$ a sequence increasing to $+\infty$, and we set
$$
\p_j(z):=\e_j \chi \circ L_j(z)-\e_j C_j,
$$
where $L_j(z):=  \log |e^{C_j} z|
=C_j+L(z)$.
The functions $\p_j$ are 
 psh in $\C^n$, since
 $$ 
 dd^c \psi_j= \varepsilon_j \chi''\circ L_j \, dL\wedge d^c L+ \varepsilon_j \chi'\circ L_j \,dd^c L.
 $$ 
Thus
$$ 
(dd^c \p_j)^k= k\varepsilon_j^k\, \chi''\circ L_j(z)(\chi'\circ L_j)^{k-1} dL\wedge d^c L\wedge (dd^c L)^{k-1}
+ \varepsilon_j^k \,(\chi'\circ L)^k (dd^c L)^k.
$$
Since $\chi''\circ L_j=0$ outside  $\mathcal{C}_j:=\{\frac{e^{-C_j}}{2}< |z|< 2e^{-C_j}\}$, we obtain
$$
(dd^c \p_j)^k\wedge \omega^{n-k}=F_{j,k}(z) dV(z), \; k =1,2,...,n,
$$
where $\omega =dd^c \log \sqrt{1+|z|^2}$ denotes the Fubini-Study form on $\mathbb{C}^n$ and 
 $$F_{j,k} =  \varepsilon_j^k  \chi''\circ L_j (\chi'\circ L_j)^{k-1} \frac{g_{k}(z)}{|z|^{2k}}
+ \varepsilon_j^k( \chi'\circ L_j)^{k} \frac{h_k(z)}{|z|^{2k}}
$$
with $g_k,h_k>0$ bounded functions. We then see that for $k=0,\cdots,n-1$ we can ensure that 
$\|F_{j,k}\|_{L^q(\mathbb{C}^n)}\leq A$ for some $q>1$ and $A>0$. 

We note that $F_{j,n}= \varepsilon_j^n  \chi''\circ L_j (\chi'\circ L_j)^{n-1} \frac{g_{n}(z)}{|z|^{2n}}$ 
 is supported on $\mathcal{C}_j$ with
$$
 F_{j,n}(z)\leq  C\e_j^n e^{2nC_j} f_n(e^{C_j}z),
 \; \; \text{ where } \; \; 
f_n(z):=|z|^{-2n} \chi''\circ L (\chi'\circ L)^{n-1}.
$$
We take $C_j=\e_j^{-n}$ so that the entropy of $F_{j,n}$ is uniformly bounded. 
Indeed 
 we see that $\int_{\mathbb{C}^n} F_{j,n}(z)\log F_{j,n}(z)dV(z)$ is comparable to 
\[
 \int_{\mathcal{C}} (2n\varepsilon_j^n C_j f_n(w)+ \varepsilon_j^n f_n(w)( \log f_n(w) + n\log \varepsilon_j)) dV(w),
\]
where $\mathcal{C}:=\{1/2<|w|<2\}$
(use the change of variables $w=e^{C_j}z$).
Now
$$
\int_{\mathcal{C} } \varepsilon_j^n f_n(w)( \log f_n(w) + n\log \varepsilon_j)) dV(w) \rightarrow 0
$$ 
 as $j\rightarrow +\infty$, while 
 $$ 
 \int_{\mathcal{C}} 2n f_n(w)dV(w) \sim \int_{\mathcal{C}} \frac{1}{|w|^{2n}} dV(w)\sim \int_{1/2}^2 \frac{1}{\rho} d\rho
 $$
  is uniformly bounded.
The same type of computations 
yields
\begin{eqnarray*}
\int_{\C^n} |\p_j|^p (dd^c \p_j)^n &\sim & \int_{\mathcal{C}_j}\e_j^p |\chi\circ L_j(z)|^p F_{j,n}(z) dV(z) +  \int_{\mathcal{C}_j}\e_j^p C_j^p F_{j,n}(z) dV(z) \\
&\sim & \e_j^{p+n}\int_{\mathcal{C}}f_n(w) dV(w)+ \e_j^{n+p}C_j^p \int_{\mathcal{C}}f_n(w) dV(w) \\
&\sim &\e_j^{n+p}C_j^p \sim \e_j^{n+p-np}  \sim 1, \quad \text{ iff } \; \; 
p=\frac{n}{n-1}.
\end{eqnarray*}
We finally consider the induced $\omega$-psh functions 
$$
\f_j=\p_j(z)- \varepsilon_j\log \sqrt{1+|z|^2}
$$
on $(\PP^n,\omega)$,
and 
conclude that
\begin{itemize}
\item $\f_j \in \Ent_B$ and $C^{-1} \leq E_{\frac{n}{n-1}}(\f_j) \leq C$;
\item $d_p(\f_j,0) \rightarrow 0$  for all $p<\frac{n}{n-1}$ but $d_{\frac{n}{n-1}}(\varphi_j,0) \not \rightarrow 0$. 
\end{itemize}
\end{exa}

 \section{The local setting}  \label{sec:local}
 We fix $\Omega \subset \C^n$  a bounded hyperconvex domain, i.e. there exists a continuous  psh function $\rho: \Omega \rightarrow [-1,0)$ such that $\{\rho <-c\}\Subset \Omega$ for all $c>0$.

\subsection{Cegrell's classes}


Let ${\mathcal T}(\Omega)$ denote the 
set of bounded non-positive psh
functions $u$ defined on $\Omega$ such that $\lim _{z\to \zeta} u (z) = 0,$ 
for every $\zeta \in \partial \Omega,$ and $\int_\Omega(dd^c u )^n <+\infty$.
Cegrell \cite{Ceg98, Ceg04} has 
studied the
complex Monge-Amp\`ere operator $(dd^c \cdot)^n$ and 
 introduced different
classes of 
plurisubharmonic functions on which the latter is well defined: 
\begin{itemize}
\item  ${\rm DMA}(\Omega)$ is the set of psh functions
 $u$ such that for all $z_0 \in \Omega$, there exists a neighborhood $V_{z_0}$ of
$z_0$ and $u_j \in {\mathcal T}(\Omega)$ a decreasing sequence which
converges to  $u$ in $V_{z_0}$ and satisfies
$\sup_j \int_{\Omega} (dd^c u_j)^n <+\infty$.

\item the class ${\mathcal F}(\Omega)$ is the ``global version'' of ${\rm DMA}(\Omega)$:
a function $u$ belongs to ${\mathcal F}(\Omega)$ iff  there exists $u_j \in {\mathcal T}(\Omega)$
a sequence decreasing towards $u$ {\it in all of } $\Omega$, which satisfies
$\sup_j \int_{\Omega} (dd^c u_j)^n<+\infty$;

\item the class ${\mathcal E}^p(\Omega)$ (respectively $\mathcal{F}^p(\Omega)$)
is the set of psh functions $u$ for which there exists a sequence of
functions $u_j \in {\mathcal T}(\Omega)$ decreasing towards $u$ in all of $\Omega$, and
so that $\sup_j \int_{\Omega} (-u_j)^p (dd^c u_j)^n<+\infty$ (respectively $\sup_j \int_{\Omega} [1+(-u_j)^p] (dd^c u_j)^n<+\infty$).
\end{itemize}

Given $u \in \mathcal{E}^p(\Omega)$ we define the weighted energy of $u$ by 
\[
E_p(u):= \int_{\Omega} (-u)^p (dd^c u)^n<+\infty.
\]  
The operator $(dd^c \cdot )^n$ is well defined on these sets, and continuous under decreasing limits. If $u\in \mathcal{E}^p(\Omega)$ for some $p>0$ then $(dd^c u)^n$ vanishes on all pluripolar sets \cite[Theorem 2.1]{BGZ09}. If $u\in \mathcal{E}^p(\Omega)$ and $\int_{\Omega} (dd^c u)^n <+\infty$ then $u\in \mathcal{F}^p(\Omega)$. Also, note that 
$$
{\mathcal T}(\Omega) \subset {\mathcal F}^p(\Omega) \subset{\mathcal F}(\Omega)\subset {\rm DMA}(\Omega)
\; \; \text{ and } \; \; 
{\mathcal T}(\Omega) \subset {\mathcal E}^p(\Omega)\subset {\rm DMA}(\Omega). 
 $$
It has been established in \cite{BGZ09}  that
\begin{equation*}
{\mathcal E}^p(\Omega)=
\left\{ \f \in \PSH^-(\Omega) \, | \,
\int_0^{+\infty} t^{n+p-1} \text{Cap}(\varphi<-t) dt <+\infty \right\}.
\end{equation*}
Here ${\rm Cap}(\cdot) := \text{Cap}(\cdot,\Omega)$ denotes the Monge-Amp\`ere capacity \cite{BT82}:
\[
{\rm Cap}(E,\Omega) := \sup \left \{ \int_E (dd^c u)^n \; | \; u \in \PSH(\Omega), \ -1\leq u \leq 0 \right \}. 
\]
Given a Borel function $h$ defined in $\Omega$ we let $P(h)$ denote the psh envelope: 
\[
P(h) := \left ( \sup\{ u \in \PSH(\Omega) \; | \; u\leq h \; \text{in}\; \Omega \} \right )^*. 
\]

\begin{lemma}
	\label{lem: local envelope}
	If $h$ is bounded and upper semicontinuous  then 
	$P(h)$ is a bounded plurisubharmonic function whose Monge-Amp\`ere measure
	$(dd^c P(h))^n$ is supported on the contact set $\{z\in \Omega \; | \; P(h)(z) =h(z)\}$. 
\end{lemma}

\begin{proof}
By assumption there exists $C>0$ such that $h\geq -C$, hence $P(h)\geq -C$.
 In particular $P(h) \in \PSH(\Omega) \cap L^{\infty}$. The fact that  $(dd^c P(h))^n$ is supported on the contact set follows from a standard balayage argument if $h$ is continuous. For the general case we let $h_j$ be a sequence of continuous functions on $\bar{\Omega}$ such that $h_j \searrow h$ in $\Omega$. Then $P(h_j)\searrow P(h)$ and 
	\[
	\int_{\Omega} (h_j-P(h_j)) (dd^c P(h_j))^n =0. 
	\]
	Also, we note that $(h_j-P(h_j))$ converges in capacity to  $(h-P(h))$.
	It follows from \cite[Theorem 4.26]{GZbook} that $ (h_j-P(h_j)) (dd^c P(h_j))^n$ 
	weakly converges to $ (h-P(h)) (dd^c P(h))^n$. Letting $j\to +\infty$ we thus obtain
    \[
    \int_{\Omega} (h-P(h)) (dd^c P(h))^n \leq \liminf_{j\to +\infty} \int_{\Omega} (h_j-P(h_j)) (dd^c P(h_j))^n =0
    \]
    which yields the desired result.
\end{proof}

\smallskip

We shall need the following maximum principle  in the sequel:
\begin{lemma}
	\label{lem: MA contact}
	Assume $u \leq v$  are bounded psh functions on $\Omega$. Then 
	\[
	{\bf 1}_{\{u=v\}} (dd^c u)^n \leq {\bf 1}_{\{u=v\}} (dd^c v)^n. 
	\]
\end{lemma}

\begin{proof} 
	It follows from \cite[Corollary 3.28]{GZbook} that 
	\[
	(dd^c \max(u,v))^n \geq {\bf 1}_{\{u\geq v\}} (dd^c u)^n +  {\bf 1}_{\{u < v\}} (dd^c v)^n,
	\]
Multiplying by ${\bf 1}_{\{u=v\}}$ yields the desired inequality.
\end{proof}

\begin{lemma}
	\label{lem: capacity sublevel Ep}
	Fix  $u \in \mathcal{T}(\Omega)$, $p> 0$ and set $q=\frac{n+p}{n+1}$. Then for all $s>0$,
	\[
	s^{n+p}\capa(u\leq -s) \leq   q^n  \int_{\Omega} (-u)^p (dd^c u)^n.
	\]
\end{lemma}
  
  \begin{proof}
  	Consider $v := P(-(-u)^q)$. Since $u$ is bounded and $q>1$, we have $Cu \leq -(-u)^q$ for some constant $C>0$. This yields $v\geq Cu$ and since $Cu \in \mathcal{T}(\Omega)$ we also have that  $v\in \mathcal{T}(\Omega)\subset \mathcal{E}^1(\Omega)$. We then set $w=-(-v)^{1/q}$ and 
  	$
  	D:= \{x\in \Omega \; |\;  v=-(-u)^q\}. 
  	$
  	Since $w\leq u$ with equality on $D$ it follows from Lemma \ref{lem: MA contact} that  
  	\[
  	{\bf 1}_D (dd^c w)^n \leq {\bf 1}_D (dd^c u)^n. 
  	\]
  	A direct computation shows that $dd^c w \geq q^{-1} (-v)^{1/q-1} dd^c v$, hence 
  	\[
  	{\bf 1}_D q^{-n} (-v)^{\frac{n(1-q)}{q}}  (dd^c v)^n \leq {\bf 1}_D(dd^c u)^n. 
  	\]
  	Since $(dd^c v)^n$ is supported on $D$ (Lemma \ref{lem: local envelope}) we infer that 
  	\[
  	(-v)(dd^c v)^n \leq {\bf 1}_D q^n (-v)^{1+ \frac{n(q-1)}{q}}  (dd^c u)^n =  {\bf 1}_D q^n (-u)^p (dd^c u)^n. 
  	 \]
  	 Integrating on $\Omega$ gives $E_1(v) \leq q^n E_p(u)$.
  	 
  	 We now use the simple fact that $\{u\leq -s\}\subseteq \{v\leq -s^q\} $ together with \cite[Lemma 2.2]{ACKPZ09}  applied to the function $v\in \mathcal{E}^1(\Omega)$ to obtain
  	 \[
  	 \capa(u\leq -s) \leq \capa (v\leq -s^q) \leq s^{-q(n+1)} E_1(v) \leq s^{-n-p}  q^nE_p(u),
  	 \]
  	 finishing the proof. 
  \end{proof}


The following is a local analogue of Proposition \ref{lem: criterion in Ep}:

\begin{prop}
	If $p>0$ and $u\in \mathcal{E}^p$ then $P(-(-u)^{1+p/n}) \in \mathcal{F}(\Omega)$.
\end{prop}

\begin{proof}
	Let $u_j$ be a sequence in $\mathcal{T}(\Omega)$ such that $u_j\searrow u$ and 
	$$
	\sup_j \int_{\Omega} (-u_j)^p (dd^c u_j)^n <+\infty.
	$$
	 Set $v_j := P(-(-u_j)^q)$, $D:= \{v_j= -(-u_j)^q\}$, where $q=1+p/n$. Since 
	 $$
	 dd^c (-(-u_j)^q)=q|u_j|^{q-1} dd^c u_j-q(q-1)|u_j|^{q-2}du_j\wedge d^c u_j \leq q|u_j|^{q-1} dd^c u_j,
	 $$
it follows from Lemma \ref{lem: local envelope} that
	\[
	(dd^c v_j)^n \leq  {\bf 1}_{D} q^n (-u_j)^{n(q-1)} (dd^c u_j)^n\leq q^n (-u_j)^{p} (dd^c u_j)^n.
	\]
Thus
$$
\sup_j \int_{\Omega}(dd^c v_j)^n \leq  q^n \sup_j E_p(u_j)<+\infty.
$$
Now $v_j \in \mathcal{T}(\Omega)$ and $v_j \searrow v: =P(-(-u)^q)$, hence $v$ belongs to $ \mathcal{F}(\Omega)$. 
\end{proof}

We will also need the following energy estimate:

 \begin{lemma}\label{lem: Ep inequality}
 Fix $p\geq 1$. If $u,v \in \mathcal{F}^p(\Omega)$ and $u\leq v$ then $E_p(v)\leq 2^{n+p}E_p(u)$.  
 \end{lemma}
 
 \begin{proof}
Observe that $\{v < -2t\}\subseteq \{2u < v-2t\} \subseteq \{u<-t \}$.Therefore
 \begin{align*}
 	E_p(v) &= p \int_0^{+\infty}  t^{p-1}(dd^c v)^n (v < -t) dt
 	 =p 2^{p}\int_0^{+\infty}  t^{p-1}(dd^c v)^n (v < -2t) dt\\
 	&\leq p 2^{p}\int_0^{+\infty}  t^{p-1}(dd^c v)^n (2u < v-2t) dt \\
 	&\leq p 2^{p+n}\int_0^{+\infty}  t^{p-1}(dd^c u)^n (2u < v-2t) dt\\
 	&\leq p 2^{n+p}\int_0^{+\infty}  t^{p-1}(dd^c u)^n (u < -t) dt
 	= 2^{n+p} E_p(u),
 \end{align*}
 where in the second inequality we have used the comparison principle \cite[Lemma 4.4]{Ceg98}. 
 \end{proof}

%
%

\subsection{Moser-Trudinger-Adams inequality}
 
\ 
 
Let $H_0^{1,2}$ be the Sobolev space of $L^2$-functions with gradient in $L^2$, 
completion of
 the space ${\mathcal D}(\Omega)$ of smooth functions with compact support in $\Omega$.
 
  The classical Moser-Trudinger inequality asserts that if $\Omega\subset \mathbb{R}^2$ has bounded area
  and $u$ belongs to the unit ball of $H_0^{1,2}$, then
  $$
  \int_{\Omega} \exp(2 \pi u^2) dV_{eucl} \leq C_1 \rm{Area}(\Omega),
  $$
 for some absolute constant $C_1>0$ (see \cite{Tru67,Mos71}). 

 This famous inequality has been generalized in various ways. Adams \cite{Ad88} notably showed that if $\Omega \subset \R^4$
 has finite volume and $u \in {\mathcal D}(\Omega)$ satisfies $\int_{\Omega} (\Delta u)^2 dV\leq 1$, then
 $$
 \int_{\Omega} \exp( 32 \pi^2 u^2) \,dV_{eucl}\leq C_2 \rm{Vol}(\Omega).
 $$
 
 Identifying $\R^2 \simeq \C$ and $\R^4 \simeq \C^2$, we propose here yet another type of generalization of these inequalities,
 replacing the average bound on $\Delta u$ by a plurisubharmonicity condition. Let us emphasize that the psh condition can be seen as a one-sided bound $dd^c u \geq 0$, that is exactly $\Delta u \geq 0$ in complex dimension $1$.
 
 \begin{thm}\label{thm: complex MTA inequality}
 Let $\Omega \subset \C^n$ be a bounded hyperconvex domain. 
 Fix $p> 0$ and  $0<\gamma<\frac{2n(n+1)}{n+p}$.
 There exists  $C_\gamma>0$ such that  
  
 \begin{equation}\label{exp_local}
 \int_{\Omega} \exp \left( \gamma E_p(u)^{-1/n} |u|^{1+\frac{p}{n}} \right) dV \leq C_\gamma
 \end{equation}
 for any non-constant $u \in {\mathcal E}^p(\Omega)$.
 \end{thm}

Theorem \ref{thm: complex MTA inequality} generalizes an important result of Berman-Berndtsson  \cite[Theorem 1.5]{BB11} which treats the case  $p=1$ and provides the  sharp constant  $\gamma=2n$ 
(beware of  the different normalization for $d^c$  in \cite{BB11}).  
The proof of \cite[Theorem 1.5]{BB11} uses induction on  dimension and  the ``thermodynamical formalism'' introduced in \cite{Ber13}. This approach does not seem to work for other energies $E_p$, $p\neq 1$.  
Similar estimates have been established by many authors using various 
 techniques (see \cite{Ceg19}, \cite{AC19},  \cite{BB14}, \cite[Theorem 1.1]{WWZ20} and the references therein).

\begin{remark}
Note that   $E_p(u)>0$ since $u\leq 0$ is non-constant.
  Indeed if $E_p(u)=0$ then $(dd^c u)^n =0$ in $\Omega$. 
  By uniqueness of solutions to the complex Monge-Amp\`ere equation \cite[Theorem 4.5]{Ceg98},  we would then get $u\equiv 0$.
\end{remark}

\begin{remark}
A scaling argument insures that the exponent $-1/n$ of the energy $E_p$ in \eqref{exp_local} is optimal.
 Indeed 
 $E_p(su)=s^{p+n} E_p(u)$ hence
 $$ E_p(su)^{-1/n} |su|^{1+\frac{p}{n}}= E_p(u)^{-1/n}|u|^{1+\frac{p}{n}}.$$
\end{remark}
\medskip

\begin{proof} We first assume that $u\in \mathcal{T}(\Omega)$. For convenience we set $A:= E_p(u)$, $q=\frac{n+p}{n+1}$ and $r= 1+ \frac{p}{n}$.   It follows from \cite[Proposition 6.1]{ACKPZ09} 
that for all $\beta<2n$,  there exists $C_{\beta}>0$ such that
  $$
 {\rm Vol}(u<-t) \leq C_{\beta}\exp\left( -\frac{\beta}{{\rm Cap}(u<-t)^{1/n}} \right)
 \leq C_{\beta}\exp\left( -\frac{\beta t^r}{ qA^{1/n}} \right),
  $$
  where the last inequality follows  from Lemma \ref{lem: capacity sublevel Ep}. 
  We infer
  \begin{eqnarray*}
  \int_{\Omega} \exp(\gamma' |u|^r) dV & =& 
  \int_0^{+\infty} r \gamma' t^{r-1} e^{\gamma' t^r} {\rm Vol} (u<-t) dt \\
  & \leq &  C_{\beta}' \int_0^{+\infty}   t^{r-1}  \exp\left( \left[\gamma'- \frac{\beta}{qA^{1/n}}\right] t^r \right) dt <+\infty
  \end{eqnarray*}
  as long as $\gamma':=\gamma  A^{-1/n}<\beta q^{-1}A^{-1/n}$, i.e. $\gamma< \frac{2n(n+1)}{n+p}$.

 	This proves the statement for $u\in \mathcal{T}(\Omega)$. 
 	We next use an approximation argument  to treat the general case of $u\in \mathcal{E}^p(\Omega)$. By definition there exists a sequence $u_j \in \mathcal{T}(\Omega)$ such that $u_j \searrow u$ and 
 	\[
 	\sup_j  E_p(u_j) := \sup_j \int_{\Omega} (-u_j)^p (dd^c u_j)^n <+\infty. 
 	\]
 	 By \cite[Theorem 3.4]{BGZ09}, $E_p(u_j) \to E_p(u)$. Using the first step and Fatou's lemma we conclude the proof. 
 \end{proof}

 In the case $p=n$ we obtain the following higher dimensional complex version of Adam's inequality:

\begin{coro}
 Let $\Omega \subset \C^n$ be a bounded hyperconvex domain. Fix
 $0<\gamma<n+1$.
 There is  $C_{\gamma}>0$ such that  
  for all $u \in {\mathcal E}^n(\Omega)$ with $E_n(u) \leq 1$, 
 $$
 \int_{\Omega} e^{\gamma  u^{2}}\, dV \leq C_{\gamma}.
 $$
\end{coro}


Arguing as in the proof of Corollary \ref{cor: MT inequality}, we obtain the following version of the Moser-Trudinger inequality: 

 \begin{corollary}
 	Fix $p>0$, $\varepsilon >0$ and $A= \frac{p}{n+p}  \left (\frac{1}{2(n+1)} \right)^{\frac{n}{p}}$.   There exists a uniform constant $B>0$ depending on $\varepsilon$ such that for all $u \in \mathcal{E}^p(\Omega)$, 
 	\[
 	\log \int_{\Omega} e^{- u} dV \leq (A+\varepsilon) E_p(u)^{\frac{1}{p}} + B.
 	\]
 \end{corollary} 

 When $p=1$ the constant $A$ is sharp as shown in \cite[Theorem 1.5]{BB11}. 
A similar result (with a less precise constant) has been established in \cite[Theorem 4.1]{AC19} 
by a different method. It was also observed in \cite{AC19} that the constant $A$ can not be smaller than $\frac{p}{n+p}  \left (\frac{1}{2(n+p)} \right)^{\frac{n}{p}}$
(beware of the different normalizations of $dd^c$ in \cite{BB11,AC19} and in the present article !).
 
 \begin{proof}
 	The proof is the same as that of Corollary \ref{cor: MT inequality}. In particular \eqref{ineq3.2} with $k=1$ gives 
 	\begin{equation}\label{ineq3.2local}
 	-u(x)  \leq  \frac{p a^{1+\frac{n}{p}} E_p(u)^{\frac{1}{p}} }{n+p} +  \frac{n (-u(x))^{1+\frac{p}{n}}}{(n+p) a^{1+\frac{p}{n}} E_p(u)^{\frac{1}{n}}}.
	\end{equation}
 	Fix  $0<c<c_0:=\frac{2n(n+1)}{n+p}$. Then 
 	\[
 	a:=  \left (\frac{n}{(n+p)c} \right)^{\frac{n}{n+p}} > \left ( \frac{1}{2(n+1)} \right )^{\frac{n}{n+p}},
 	\]
 	and 
 	\[
 	A':= \frac{p a^{1+ \frac{n}{p}}}{n+p} > \frac{p}{n+p}  \left (\frac{1}{2(n+1)} \right)^{\frac{n}{p}}=A, 
 	\]
 	with equalities when $c=c_0$. The inequality \eqref{ineq3.2local} 
 	can then be rewritten as
 	$$ 	-u(x)  \leq A' E_p(u)^{\frac{1}{p}} +  cE_p(u)^{-\frac{1}{n}} (-u(x))^{1+\frac{p}{n}}.$$
Now Theorem \ref{thm: complex MTA inequality} ensures
that
 	\[
 	\log \int_{\Omega} e^{- u} dV \leq A' E_p(u)^{\frac{1}{p}} + B. 
 	\]
 For any $\varepsilon>0$ we can choose $c$ so close to $\frac{2n(n+1)}{n+p}$ that $A' \leq A+ \varepsilon$.
 \end{proof}

\subsection{Finite entropy potentials}

 Let $\mu= fdV \geq 0$ be a probability measure on $\Omega$ with finite entropy, i.e. 
  $\Ent(f) :=\int_{\Omega} f \log f dV<+\infty$. 
  Cegrell \cite[lemma 5.14]{Ceg04}  has shown that there exists a unique psh function $\varphi \in \mathcal{F}(\Omega)$ such that $(dd^c \varphi)^n =\mu$.  
We show here that  $\varphi$ belongs to an appropriate finite energy class:

 \begin{thm}\label{thm: entropy local}
The function  $\varphi$  belongs to ${\mathcal E}^{p}(\Omega)$ for all $0<p\leq \frac{n}{n-1}$. 
Moreover, there exist $c, C>0$  depending on $n$, $p$, $\Omega$ and $\Ent(f)$ such that 
$$
E_p(\varphi) \leq C
\; \; \text{ and } \; \;
 \int_{\Omega} e^{c |\f|^p} \omega^n \leq C.
$$
  \end{thm}
 
Let us stress that the RHS integral estimate has been obtained with completely different methods 
by Wang-Wang-Zhou \cite[Theorem 3.2]{WWZ20}.

  \begin{proof}
By H\"older inequality it suffices to prove the result for $p=\frac{n}{n-1}$. 
We approximate $f$ by $f_j := \min(f,j)$ and observe that $f_j dV\leq dd^c(bj^{1/n}|z|^2)^n$, 
where $b>0$ is a normalization constant such that $dV=b^n\,(dd^c|z|^2)^n$. 

Note that for each $j\in \N$, $f_j$ still has finite entropy and $\int_{\Omega} f_j dV<+\infty$. 
By \cite[Proposition 6.1]{Ceg98}
there exists a unique $\varphi_j \in \mathcal{F}^1(\Omega)\cap L^\infty(\Omega)$  such that  $(dd^c \varphi_j)^n = f_jdV$.
The comparison principle \cite[Theorem 4.5]{Ceg98}
insures that $j \mapsto \varphi_j$ is decreasing and
the same argument as in the proof of Theorem \ref{thm: complex MTA inequality}  shows that $\varphi_j$ decreases to $\varphi$. 
Now Theorem \ref{thm: complex MTA inequality} yields
\[
\int_{\Omega} \exp (\gamma  E_p(\varphi_j)^{-1/n} |\varphi_j|^p) dV  \leq C_1,
\]
where $C_1$ depends on $\gamma$. 
Applying H\"older-Young inequality \eqref{eq:conj}  we obtain
\begin{eqnarray*}
&&\int_{\Omega} \gamma E_p(\varphi_j)^{-1/n} |\varphi_j|^p f_j dV  \\
&&\leq \int_{\Omega} \left(e^{\gamma E_p(\varphi_j)^{-1/n} |\varphi|^p}- \gamma E_p(\varphi_j)^{-1/n} |\varphi_j|^p-1 \right)dV\\
&&\quad + \int_{\Omega} (f_j+1)\log (f_j+1) dV-\int_{\Omega} f_j dV
\leq C_2,	
\end{eqnarray*}
where $C_2$ depends on $C_1$ and $\Ent(f)$. 
Thus $\gamma E_p(\varphi_j)^{1-1/n} \leq C_2$, hence $E_p(\varphi_j) \leq C_3$ where $C_3$ depends on $C_1$ and $\Ent(f)$.  Using \cite[Theorem 3.4]{BGZ09} it thus follows that $E_p(\varphi)\leq C_3$. 
Combining the above  upper bound with Theorem \ref{thm: complex MTA inequality} yields
$$
 \int_\Omega \exp{(\gamma C_2^{-1/n} |\varphi|^p )}\, dV \leq C,
 $$
as desired.
\end{proof}  


We finally establish a local analogue of Theorem \ref{thm: compact injection}:

\begin{theorem}
Let $0\leq f_j$ be a sequence of densities with uniformly bounded entropy. 
Let $u_j \in \mathcal{F}^{\frac{n}{n-1}}(\Omega)$ be the unique solutions to $(dd^c u_j)^n = f_j dV$
and fix $0<p<\frac{n}{n-1}$.
There exist $u\in \mathcal{F}^p(\Omega)$ and a subsequence, still denoted by $u_j$, such that  $\|u_j-u\|_{L^1(\Omega)} \to 0$ and
\begin{equation}\label{eq: dp convergence local}
  \lim_{j\to +\infty}\int_{\Omega} |u_j-u|^p ((dd^c u_j)^n +(dd^c u)^n) = 0. 
\end{equation}
 In particular $(dd^c u_j)^n$ weakly converges to $(dd^c u)^n$. 
\end{theorem} 

 Example \ref{ex: inj_nocompact} shows that \eqref{eq: dp convergence local} does not hold for $p=\frac{n}{n-1}$.
 
 \begin{proof}
 Fix $0<p< r:=\frac{n}{n-1}$.
 It follows from Theorem \ref{thm: entropy local} that $E_{r}(u_j)$ and  $\int_{\Omega}e^{c|u_j|^r}$ are uniformly bounded.
 Thus up to extracting and relabelling, $u_j$ converges in $L^{1}_{{\rm loc}}$ and a.e. to $u\in \PSH(\Omega)$.  If we set $v_k:= (\sup_{j\geq k} u_j)^*$ then $v_k \searrow u$ and by Lemma \ref{lem: Ep inequality} $E_r(v_k) \leq 2^{n+r}E_r(u_k)$ is uniformly bounded.  It then follows from \cite[Theorem 3.4]{BGZ09} that $u\in \mathcal{F}^r(\Omega)$. The H\"older inequality then ensures that $u\in \mathcal{F}^p(\Omega)$ for any $p<r$. It follows from Theorem \ref{thm: complex MTA inequality} that, for some constant $\gamma>0$, $\int_{\Omega} (e^{\gamma |u_j|^r} +e^{\gamma |u|^r} )dV$ is uniformly bounded. It then follows that $\int_{\Omega}|u_j-u| dV \to 0$.
 

 It follows from \cite[Theorem 3.4]{Per99} that for all $\varphi,\psi \in \mathcal{T}(\Omega)$, 
\begin{equation*}
\int_{\Omega} |\varphi|^r (dd^c \psi)^n \leq D_{n,r} E_r(\varphi)^{r/(n+r)} E_r(\psi)^{n/(n+r)}.
\end{equation*}
By an approximation procedure one can show that the above inequality holds for $\varphi,\psi \in \mathcal{F}^r(\Omega)$ as well. 
Using this for $u_j$ and $u$ we conclude that $\int_{\Omega} |u_j-u|^r ((dd^c u_j)^n+ (dd^c u)^n)$ is uniformly bounded.

We next prove that $\int_{\Omega} |u_j-u|^p(dd^c u_j)^n \to 0$. Fixing $\varepsilon>0$, by Egorov's theorem there exists a Borel subset $G\subset \Omega$ with $ {\rm Vol}(G)<\varepsilon$ such that $u_j$ converges uniformly to $u$ in $\Omega\setminus G$. We then have 
\[
\int_{\Omega \setminus G} |u_j-u|^p f_j dV  \to 0
\]    
since $\int_{\Omega} (dd^c u_j)^n dV$ is uniformly bounded. By H\"older's inequality we have 
\[
\int_{G} |u_j-u|^pf_j  dV  \leq  \left (\int_{G} |u_j-u|^r f_j dV \right)^{p/r}  \left( \int_{G}f_j dV \right)^{q} 
\]
with $q:= \frac{r}{r-p}$. The first factor is uniformly bounded thanks to the above. Fix now $t>1$. We estimate the second factor as follows: 
\[
\int_{G \cap \{f_j \leq t\}} f_j dV \leq t  {\rm Vol}(G) \leq t \varepsilon,
\]
\[
\int_{G\cap \{f_j > t\}} f_j dV \leq \frac{1}{\log t} \int_{\Omega} f_j \log f_j dV \leq \frac{C}{\log t}.
\]
Letting $\varepsilon\to 0$ and then $t\to +\infty$ we see that $\int_{\Omega} |u_j-u|^p (dd^c u_j)^n \to 0$. 
From this and \cite[Lemma 5.3]{Ceg98} we deduce that $(dd^c u_j)^n$ weakly converges to $(dd^c u)^n$. Note that in \cite[Lemma 5.3]{Ceg98}  it was assumed that $u_j\in \mathcal{T}(\Omega)$ is continuous for all $j$ but the proof does apply to our more general setting.  This together with  lower semicontinuity of the entropy reveal that $(dd^c u)^n$ also has finite entropy. We can thus repeat the same arguments as above to conclude that  $\int_{\Omega} |u_j-u|^p (dd^c u)^n \to 0$, finishing the proof.
\end{proof}


\end{document}